\documentclass[12pt]{article}
\usepackage[left=2cm,right=2.3cm,
    top=2cm,bottom=2cm,bindingoffset=0cm]{geometry}
\usepackage{amsthm}
\usepackage{graphicx,textcomp}
\usepackage{amsmath,bm,amsfonts,mathrsfs,amssymb,amsbsy}
\usepackage{hyperref}
\newtheorem{theorem}{Theorem}[section]
\newtheorem{lemma}[theorem]{Lemma}

\newtheorem{cor}[theorem]{Corollary}
\usepackage{appendix}

\providecommand{\keywords}[1]
{
  \small	
  {\textit{Keywords: }} #1
}

\begin{document}
\title{Determinants of pseudo-laplacians and $\zeta^{({\rm reg})}(1)$ for spinor bundles over Riemann surfaces}
\author{Alexey Kokotov\thanks {\href{mailto:alexey.kokotov@concordia.ca}{\nolinkurl{alexey.kokotov@concordia.ca}}}, Dmitrii Korikov\thanks {\href{mailto:dmitrii.v.korikov@gmail.com}{\nolinkurl{dmitrii.v.korikov@gmail.com}}}.}
\maketitle
\begin{abstract}
Let $P$ be a point of a compact Riemann surface $X$. We study self-adjoint extensions of the Dolbeault Laplacians in hermitian line bundles $L$ over $X$ initially defined on sections with compact supports in $X\backslash\{P\}$. We define the $\zeta$-regularized determinants for these operators and derive comparison formulas for them. We introduce the notion of the Robin mass of $L$. This quantity enters the comparison formulas for determinants and is related to the regularized $\zeta(1)$ for the Dolbeault Laplacian. For spinor bundles of even characteristic, we find an explicit expression for the Robin mass. In addition, we propose an explicit formula for the Robin mass in the scalar case. Using this formula, we describe the evolution of the regularized $\zeta(1)$ for scalar Laplacian under the Ricci flow. As a byproduct, we find an alternative proof for the Morpurgo result that the round metric minimizes the regularized $\zeta(1)$ for surfaces of genus zero.
\end{abstract}


\section{Introduction}

Let $X$ be a compact Riemann surface of genus $g$ endowed with smooth conformal metric $\rho^{-2}|dz|^2$ and let $L$ be a holomorphic line bundle over $X$ with smooth hermitian metric $h$. The Dolbeault Laplacian $\Delta$ acts on smooth sections of $L$ by
\begin{equation}
\label{Dol}
\Delta u=-4\rho^2 h^{-1}\partial (h\overline{\partial}u).
\end{equation}
Its closure in $L_2(X;L)$ is a self-adjoint operator, also denoted by $\Delta$. 

Let $x,y,z$ denote holomorphic local coordinates on $X$ and let $x\mapsto u(x)$ denote the representative of a section $u$ in a local coordinate $x$. Let $x,y\mapsto G(x,y)$ be the Green function (section of $L_x\hat{\otimes}\overline{L_y}$) of $\Delta$ and $x,y\mapsto\mathsf{1}(x,y)$ be a smooth section of $L_x\hat{\otimes}L^{-1}_{y}$ obeying $\mathsf{1}(x,x)=1$.

Chose a point $P$ of $X$. Introduce the operator $\dot{\Delta}$ as the $L_2(X;L)$-closure of operator (\ref{Dol}) defined on smooth sections of $L$ with compact supports in $\dot{X}=X\backslash\{P\}$. In Section \ref{SA EXT sec}, we prove that the operators $\Delta_\alpha$ ($\alpha\in(-\pi/2,\pi/2]$) acting via
\begin{equation}
\label{Domains act}
\Delta_\alpha u=\Delta\big(u-c_u(y)h(y)G(\cdot,y){\rm sin}\alpha\big) \qquad (y=y(P))
\end{equation}
on the domains
\begin{equation}
\label{Domains}
{\rm Dom}\Delta_\alpha=\big\{u=c_u(y)\big(h(y)G(\cdot,y){\rm sin}\alpha+\mathsf{1}(\cdot,y){\rm cos}\alpha\big)+\tilde{u} | \ y=y(P), c_u\in \Gamma(X;L), \tilde{u}\in {\rm Dom}\dot{\Delta}\big\}
\end{equation}
are all the self-adjoint extensions of $\dot{\Delta}$ while $\Delta_0\equiv \Delta$ is the Friedrichs extension. 

This statement extends the result of Colin de Verdi\`ere \cite{CdV} who dealt with case of trivial bundle $L$ with $h=1$. Following \cite{CdV}, we call the operators $\Delta_\alpha$ ($\alpha\ne 0$) {\it pseudo-laplacians}. The (scalar) pseudo-laplacians arise as rigorous counterparts of the formal operators $\Delta+\epsilon\delta_{P}$ (where $\delta_{P}$ is the Dirac measure at $P$ and $\epsilon\in\mathbb{R}$, see \cite{BFad} and Chapter III,4 \cite{Albe}) in the models of point scattering of quantum particles first introduced by Enrico Fermi \cite{Fermi}. The equation $\Delta_\alpha u=\lambda u$ (in the scalar case) describes motion of a quantum particle on the surface in the presence of a point scatterer (S\v{e}ba billiard, see \cite{Seba}).

Our main goal is to study the $\zeta$-regularized determinants of $\Delta_\alpha$. In Section \ref{Comp F sec}, we derive comparison formulas for the determinants of $\Delta_\alpha$ and $\Delta$. From now on, we assume that $L$ admits no non-trivial holomorphic sections, $h^0(L)=0$; then ${\rm Ker}\Delta=\{0\}$. Except for the case $\alpha=0$, the derivative of zeta-function $s\mapsto\zeta(s|\Delta_\alpha)$ has logarithmic singularity at $s=0$. For this reason, we apply the following regularization (proposed and discussed in a similar context by Kirsten, Loya, and Park, see \cite{KLP})
\begin{equation}
\label{reg zeta}
{\rm det}^{(r)}\Delta_\alpha={\rm exp}\big(-\partial_s \zeta^{(r)}(s|\Delta_\alpha)\big)|_{s=0}, \qquad \zeta^{(r)}(s|\Delta_\alpha)=\zeta(s|\Delta_\alpha)+s{\rm log}s
\end{equation}
for the determinants of pseudo-laplacians $\Delta_\alpha$. We prove that 
\begin{align}
\label{comparison formula determinants}
\frac{{\rm det}^{(r)}\Delta_\alpha}{{\rm det}\Delta}=-4\pi e^\gamma{\rm ctg}\alpha. 
\end{align}
where $\gamma$ is the Euler constant. 

It might seem that the dependence of ${\rm det}^{(r)}\Delta_\alpha/{\rm det}\Delta$ on the surface $(X,\rho^{-2})$, the bundle $(L,h)$ and the point $P$ in formula (\ref{comparison formula determinants}) is trivial. However, such a dependence is included implicitly in our parameterization (\ref{Domains act}), (\ref{Domains}) of pseudo-laplacians (i.e. in the way to assign a number $\alpha$ to a self-adjoint extension of $\dot{\Delta}$). Parameterization (\ref{Domains act}), (\ref{Domains}) is coordinate independent in the sense that the pseudo-laplacian is described in terms of the invariantly (and globally) defined Green section $G$.  

One can parametrize the pseudo-laplacians in a purely local way by describing the asymptotics of sections from their domains near $P$. However, such a parametrization is obtained at the cost of the loss of coordinate independence (except of the case of the trivial bundle). Namely, let $x$ be a holomorphic coordinate in the neighborhood of $P$ and let $y=x(P)$. Introduce the operator $\Delta^{(\beta)}$ acting as
\begin{equation}
\label{domain via asymptotics act}
\Delta^{(\beta)}u=-4\rho^2 h^{-1}\partial (h\overline{\partial}u) \text{ in } \dot{X}
\end{equation}
on all the sections $u$ of $L$ that are locally $H^2$-smooth outside $P$ and admit the asymptotics
\begin{equation}
\label{domain via asymptotics}
u(x)={\rm cos}\beta+{\rm sin}\beta\Big[-\frac{1}{2\pi}{\rm log}|x-y|+\frac{\partial_y{\rm log} h(y)}{4\pi}(x-y){\rm log}(\overline{x-y})\Big]+\tilde{u}
\end{equation}
near $P$, where $\tilde{u}$ is $H^2$-smooth in a neighborhood of $P$ and $\tilde{u}(y)=0$. Then $\Delta^{(\beta)}$ with $\beta\in (-\pi/2,\pi/2)$ are all the self-adjoint extensions of $\dot{\Delta}$.

To describe the relation between different parametrizations $\Delta_\alpha$, $\Delta^{(\beta)}$ of the pseudo-laplacian, let us recall that the Green function $G$ admits the asymptotics 
\begin{equation}
\label{Green asymp}
\begin{split}
h(y)G(x,y)=-&\frac{1}{2\pi}{\rm log}\,d(x,y)+ m(y)+o(1)=\\
=-&\frac{1}{2\pi}{\rm log}\,|x-y|+\frac{1}{2\pi}{\rm log}\rho(y)+ m(y)+o(1) \qquad (x\to y)
\end{split}
\end{equation}
where $d(x,y)$ is the distance between $x$ and $y$ in the metrics $\rho^{-2}|dz|^2$. In the case of trivial bundle $L$ with $h=1$, the coefficient $m(y)$ in (\ref{Green asymp}) is called the {\it Robin mass} at $y$ (see, e.g., \cite{Nev,Steiner,Okikiolu}). Similarly, we call $m(y)$ in (\ref{Green asymp}) the {\it Robin mass} at $y$ associated with the Riemannian manifold $(X,\rho^{-2})$ and the hermitian line bundle $(L,h)$. (Note that $y\mapsto m(y)$ is a scalar function on $X$, cf. p.199, \cite{Nev}.) Comparing (\ref{Domains}) and (\ref{domain via asymptotics}) by using (\ref{Green asymp}), we obtain
\begin{equation}
\label{parametrization comparing}
\Delta^{(\beta)}=\Delta_\alpha \ \Longleftrightarrow \ {\rm ctg}\beta={\rm ctg}\alpha+m(P)+\frac{1}{2\pi}{\rm log}\rho(y). 
\end{equation}
Thus, formula (\ref{comparison formula determinants}) can be rewritten as
\begin{align}
\label{comparison formula determinants coordepend}
\frac{{\rm det}^{(r)}\Delta^{(\beta)}}{{\rm det}\Delta}=-4\pi e^\gamma\Big({\rm ctg}\beta-m(P)-\frac{1}{2\pi}{\rm log}\rho(y)\Big). 
\end{align}
Here the pseudo-laplacian is defined in purely local terms (such as the metrics and their derivatives at $P$ in coordinate $x$) while the dependence of ${\rm det}^{(r)}\Delta^{(\beta)}/{\rm det}\Delta$ on $(X,\rho^{-2})$, $(L,h)$, $P$ becomes explicit due to the presence of the Robin mass $m(P)$ in the right-hand side. Equality (\ref{comparison formula determinants coordepend}) extends the result obtained for the case of trivial bundle $L$ in \cite{AissHillKok}, Theorem 1. Note that, in the case of trivial bundle $L$, $\lambda=0$ is an eigenvalue of $\Delta$ ($h^0(L)=1$) and the analogue of formula (\ref{comparison formula determinants}) is less interesting.

To make formula (\ref{comparison formula determinants coordepend}) completely explicit one needs to calculate the Robin mass $m(P)$. In Section \ref{sec Robin} we compute $m(P)$ for holomorphic line bundles $L$ obeying
\begin{equation}
\label{bundle restrictions}
{\rm deg}(L)=g-1, \quad h^0(L)=0.
\end{equation}
In particular, for a generic surface $X$, this includes the case of spin-$\frac{1}{2}$ bundles with even characteristics. In the latter case, formula (\ref{comparison formula determinants coordepend}) becomes completely explicit since ${\rm det}\Delta$ in the left-hand side is related to the scalar Laplacian via the Bost--Nelson bosonization formula (see \cite{BN}) while plenty of explicit formulas for determinants of scalar Laplacians are available. 

Note that each $L$ obeying (\ref{bundle restrictions}) is isomorphic to ${\bm\triangle}\otimes\chi$, where ${\bm\triangle}$ is the basic (i.e. with characteristic $(0,0)$) spinor bundle while $\chi$ is a unitary holomorphic line bundle (see Example 2.3 on pp.28,29, \cite{Fay}). We prove the following formula (for the case $g\ge 2$)
\begin{align}
\label{Robin mass explicit}
\begin{split}
m(x)=\frac{1}{4\pi^2}\int\limits_{X}\left[\Big|\frac{\theta[\chi]\big(\mathcal{A}(y-x)\big)}{\theta[\chi](0)E(x,y)}\Big|^2 
\frac{h(x)}{h(y)}-\Big|\partial_y{\rm log}\left(\frac{F(x,y)}{\rho(x)\rho(y)}\right)\Big|^2\right.+&\\
+\left.\Big(\frac{K(y)}{4\rho^{2}(y)}-\pi(\overline{\vec{v}(y)})^t (\Im\mathbb{B})^{-1}\vec{v}(y)\Big){\rm log}\left(\frac{F(x,y)}{\rho(x)\rho(y)}\right)\right]&\hat{d}y.
\end{split}
\end{align}
Here $\hat{d}y=d\overline{y}\wedge dy/2i$, $E(x,y)$ is the prime-form of $X$, $\theta(\chi)$ is the theta-function (defined in \cite{Fay}, (1.9)), $\vec{v}=(v_1,\dots,v_g)^t$ is the basis of Abelian differentials on $X$ normalized with respect to a chosen canonical basis of cycles, $\mathbb{B}$ is the matrix of $b$-periods of $X$, $\mathcal{A}(\mathscr{D})$ denotes the Abel transform of the divisor $\mathscr{D}$ on $X$, and $K(y)=[4\rho^2\partial\overline{\partial}{\rm log}\rho](y)$ is the Gaussian curvature of the metric $\rho^{-2}|dz|^2$ at $y$. The symmetric section 
\begin{equation}
\label{F section}
F(x,y)={\rm exp}\Big[-2\pi\Im\mathcal{A}(x-y)^t (\Im\mathbb{B})^{-1}\Im\mathcal{A}(x-y)\Big]|E(x,y)|^2
\end{equation}
of $|K_x|^{-1}\hat{\otimes}|K_y|^{-1}$ has been introduced by E. and H. Verlinde (see \cite{Ver}), formula (5.10). Note that
\begin{equation}
\label{Green Verlinde}
G^{(sc)}(x,y)=\frac{1}{2}\big(m^{(sc)}(x)+m^{(sc)}(y)\big)-\frac{1}{4\pi}{\rm log}\left(\frac{F(x,y)}{\rho(x)\rho(y)}\right),
\end{equation}
where $G^{(sc)}(x,y)$ and $m^{(sc)}$ are the scalar Green function and the Robin mass, respectively (see \cite{Ver}, formula (5.7)). Thus, as emphasized in \cite{Ver}, $F(x,y)$ can be considered as a conformally invariant part of the scalar Green function. 

Note that the integrand in the right-hand side of (\ref{Robin mass explicit}) contains two nonintegrable terms terms whose singularities (of order $|x-y|^{-2}$) cancel. Thus, the whole integrand has only integrable singularity (of order $O(|x-y|^{-1})$).

It should be noted that the Robin mass and the zeta function $s\mapsto \zeta(s|\Delta)$ of the Laplacian $\Delta$ are related as follows. Recall that $s\mapsto \zeta(s|\Delta)$ has a simple pole at $s=1$. One can define the regularized $\zeta(1|\Delta)$ as
\begin{equation}
\label{reg zeta 1}
\zeta^{(r)}(1|\Delta)=\lim_{s\to 1}\Big(\zeta(s|\Delta)-\frac{{\rm Area}(X;\rho)}{4\pi(s-1)}\Big).
\end{equation}
Then 
\begin{equation}
\label{ADM via Robin}
\zeta^{(r)}(1|\Delta)=\int\limits_{X}m(x)dS_\rho(x)+\frac{\gamma-{\rm log}2}{2\pi}{\rm Area}(X;\rho),
\end{equation}
where $dS_\rho(x)=\rho^{-2}(x)d\overline{x}\wedge dx/2i$ is the volume form on $X$ and ${\rm Area}(X;\rho)$ is the area of $X$ in the metric $\rho^{-2}|dz|^2$. Formula (\ref{ADM via Robin}) is derived in \cite{Steiner}, Proposition 2 for the case of the scalar Laplacian $-4\rho^2\partial\overline{\partial}$. In Section \ref{Steiner sec} we extend the proof of (\ref{ADM via Robin}) to the general case by making use of the results of \cite{Fay}. 

Expressions (\ref{F section}), (\ref{Green Verlinde}) turn out to be useful for the study of $\zeta^{(r)}(1|\Delta)$ in the case of the scalar Laplacian $\Delta=\Delta^{(sc)}=-4\rho^2\partial\overline{\partial}$. In Section \ref{Sec Rissi}, we derive explicit formula (\ref{scalar Robin}) for the Robin mass $m^{(sc)}$ in the scalar case. Using this formula, we describe the evolution (given by equations (\ref{Ricci evolution of RM}) and (\ref{ADM via Robin})) of $\zeta^{(r)}(1|\Delta^{(sc)})$ for scalar Laplacian under the Ricci flow. In the genus zero case, we prove that $\zeta^{(r)}(1|\Delta^{(sc)})$ is non-increasing under the Ricci flow. As a byproduct, we find an alternative proof for the Morpurgo result that the round metrics minimizes the regularized $\zeta^{(r)}(1|\Delta^{(sc)})$ for surfaces of genus zero.

\section{Pseudo-laplacians}
\label{SA EXT sec}
In this section, we describe all the self-adjoint extensions of $\dot{\Delta}$. 

First, let us describe the domain of $\dot{\Delta}$. Let $\{U_k,z_k\}_{k}$ be a finite biholomorphic atlas on $X$ and let $\{\phi_k\}$ be a (smooth) partition of unity on $X$ subordinate to the open cover $\{U_k\}_k$. We assume that $U_1$ is a neighborhood of $P$, $z_1(P)=0$, and the support of $\phi_1$ is sufficiently small. In what follows, we denote $\xi^1=\Re z_1$, $\xi^2=\Im z_1$, and $r=|z_1|$. Introduce the Sobolev space $H^l(X;L)$ of sections of $L$ with finite norms
\begin{equation}
\label{Sobolev norm}
\|u\|_{H^l(X;L)}=\left(\sum_{k}\|\ \acute{u}_k\|^2_{H^l(\mathbb{C})}\right)^{\frac{1}{2}}, \qquad \acute{u}_k(z_k):=\phi_k(z_k) u(z_k)
\end{equation}
(here $H^l(\mathbb{C})$ is the usual Sobolev space and $\acute{u}_k=0$ outside $z_k(U_k)$). 

Let us recall well-known properties of $H^l(X;L)$. Smooth sections of $L$ are dense in any $H^l(X;L)$. Since operator (\ref{Dol}) is elliptic, the $H^2(X;L)$-norm is equivalent to the graph norm 
\begin{equation}
\label{Greph norm}
\|u\|_{\Delta}=\big(\|\Delta u\|^2_{L_2(X;L)}+\|u\|^2_{L_2(X;L)}\big)^{\frac{1}{2}}
\end{equation}
of $\Delta$. The embedding $H^2(X;L)\subset C(X;L)$ is continuous. In view of the last property, sections $u\in H^2(X;L)$ vanishing at $P$ constitute the subspace $H_0^2(\dot{X};L)$ in $H^2(X;L)$. Let $\dot{X}=X\backslash\{P\}$, let $C^{\infty}_0(\dot{X};L)$ be the space of all smooth sections of $L$ vanishing at $P$, and let $C^{\infty}_c(\dot{X};L)$ be the space of all smooth sections of $L$ with compact supports in $\dot{X}$. 
\begin{lemma}
\label{H2 space lemma}
$C^{\infty}_c(\dot{X};L)$ is dense in $H_0^2(\dot{X};L)$.
\end{lemma}
\begin{proof}
Let $u\in H_0^2(\dot{X};L)$. Then there is a sequence of smooth sections $u_k$ converging to $u$ in $H^2(X;L)$. Due to the continuity of the embedding $H^2(X;L)\subset C(X;L)$, the last convergence implies $u_k(z(P))\to 0$. Then the sections $v_k=u_k-u_k(z(P))\mathsf{1}(\cdot,z(P))$ converge to $u$ in $H^2(X;L)$ while $v_k(z(P))=0$. Thus, $C^{\infty}_0(\dot{X};L)$ is dense in $H_0^2(\dot{X};L)$.

Next, suppose that $u\in C^{\infty}_0(\dot{X};L)$. Let $\kappa\in C^{\infty}(\mathbb{R})$, $\chi(s)=0$ for $s\le 0$ and $\chi(s)=1$ for $s\ge 1$. For large $N$, introduce the cut-off function $\kappa_{N}$ on $X$ which is defined by
$$\kappa_{N}(z_1)=\kappa({\rm log}\,|{\rm log}\,r|-N)$$
on $U_1$ and is equal to one outside $U_1$. Then each section $u^{(N)}=\kappa_{N}u$ belongs to $C^{\infty}_c(\dot{X};L)$ while $\|u-u^{(N)}\|_{H^2(X;L)}=\|(1-\kappa_{N})\acute{u}_1\|_{H^2(\mathbb{C}))}$. Since $\acute{u}_1$ is smooth and $\acute{u}_1(0)=0$, we have
$$\int\limits_{\mathbb{C}}|\partial^2_r[(1-\kappa_{N})\acute{u}_1]|^2d\xi^1d\xi^2\le \int\limits_{{\rm log}(-s)=N}^{N+1}\frac{cr^2}{r^4({\rm log}\,r)^2}rdr=\int\limits_{-e^N}^{-e^{N+1}}\frac{c}{s^2}ds \to 0 \quad (N\to +\infty).$$
This and similar estimates for other partial derivatives of $(1-\kappa_{N})\acute{u}_1$ yield $\|(1-\kappa_{N})\acute{u}_1\|_{H^2(\mathbb{C}))}\to 0$ as $N\to +\infty$. Therefore, $u^{(N)}\to u$ in $H^2(X;L)$ and $C^{\infty}_c(\dot{X};L)$ is dense in $H_0^2(\dot{X};L)$.
\end{proof}
Since $C^{\infty}_c(\dot{X};L)\subset {\rm Dom}\dot{\Delta}$ and the $H^2(\dot{X};L)$-norm and the graph norm are equivalent, Lemma \ref{H2 space lemma} implies the following corollary.
\begin{cor}
\begin{equation}
\label{domain of delta dot}
{\rm Dom}\dot{\Delta}=H_0^2(\dot{X};L).
\end{equation}
\end{cor}
Now, let us describe the domain of $\dot{\Delta}^*$.
\begin{lemma}
We have
\begin{equation}
\label{domain of adjoint}
{\rm Dom}\dot{\Delta}^*=\big\{u=C(y)h(y)G(\cdot,y)+c(y)\mathsf{1}(\cdot,y)+\tilde{u} \ | \ y=y(P), \ c,C\in \Gamma(X;L), \ \tilde{u}\in {\rm Dom}\dot{\Delta}\big\},
\end{equation}
where $y$ is a holomorphic coordinate in a neighborhood of $P$.
\end{lemma}
\begin{proof}
Suppose that $\dot{\Delta}^*u=f$, i.e. $(u,\Delta v)_{L_2(X;L)}=(u,\dot{\Delta} v)_{L_2(X;L)}=(f,v)_{L_2(X;L)}$ for any $v\in {\rm Dom}\dot{\Delta}\equiv H_0^2(\dot{X};L)\subset {\rm Dom}\Delta$ (see (\ref{domain of delta dot})). Any section $v\in H^2(X;L)\equiv {\rm Dom}\Delta$ can be represented as $v(x)=v(y)\mathrm{1}(x,y)+\tilde{v}(x)$, where $y=y(P)$ is a holomorphic coordinate of $P$ while $\tilde{v}\in H_0^2(\dot{X};L)$. We have
\begin{align*}
(u,\Delta v)_{L_2(X;L)}-\overline{v(y)}(u,\Delta\mathrm{1}(\cdot,y))_{L_2(X;L)}&=\\
=(u,\Delta \tilde{v})_{L_2(X;L)}=(f,\tilde{v})_{L_2(X;L)}&=(f,v)_{L_2(X;L)}-\overline{v(y)}(f,\mathrm{1}(\cdot,y))_{L_2(X;L)},
\end{align*}
i.e.
$$(u,\Delta v)_{L_2(X;L)}-\overline{v(y)}C(y)=(f,v)_{L_2(X;L)},$$
where $C(y)=(u,\Delta\mathrm{1}(\cdot,y))_{L_2(X;L)}-(f,\mathrm{1}(\cdot,y))_{L_2(X;L)}$. Recall that 
\begin{equation}
\label{bergman}
(G(\cdot,y),\Delta v)_{L_2(X;L)}+(B(\cdot,y),v)_{L_2(X;L)}=\overline{v(y)},
\end{equation}
where $(x,y)\mapsto B(x,y)$ is the Bergman kernel of $\Delta$ (the integral kernel of the orthogonal projection on ${\rm Ker}\Delta$ in $L_2(X;L)$). Thus, 
$$(u-C(y)G(\cdot,y),\Delta v)_{L_2(X;L)}=(f+C(y)B(\cdot,y),v)_{L_2(X;L)} \qquad \forall v\in H^2(X;L),$$
i.e. $u-C(y)G(y,\cdot)\in {\rm Dom}\Delta^*={\rm Dom}\Delta=H^2(X;L)$ and 
$$\Delta^*[u-C(y)G(\cdot,y)]=f+C(y)B(\cdot,y).$$
In particular, there is $c(y)$ such that $u(x)-C(y)G(x,y)=c(y)\mathrm{1}(x,y)+\tilde{u}$, where $\tilde{u}\in H_0^2(\dot{X};L)={\rm Dom}\dot{\Delta}^*$ due to (\ref{domain of delta dot}).
\end{proof}
Now we describe the self-adjoint extensions of $\dot{\Delta}$. 
\begin{lemma}
The operators $\Delta_\alpha$ ($\alpha\in(-\pi/2,\pi/2]$) defined by {\rm (\ref{Domains act})},{\rm(\ref{Domains})} are all the self-adjoint extensions of $\dot{\Delta}$. The Friedrichs extension of $\dot{\Delta}$ is $\Delta_0\equiv\Delta$.
\end{lemma}
\begin{proof}
According to (\ref{domain of adjoint}), the map
$$u=C_u(y)h(y)G(\cdot,y)+c_u(y)\mathsf{1}(\cdot,y)+\tilde{u}\mapsto \Big (C_u(y),c_u(y)\Big)$$
induces the isomorphism between ${\rm Dom}\dot{\Delta}^*/{\rm Dom}\dot{\Delta}$ and $\Gamma(X;L)/C_0^{\infty}(\dot{X};L)\simeq\mathbb{C}$. The equation
\begin{equation}
\label{symp form expl}
\mathrm{S}(u_1/{\rm Dom}\dot{\Delta},u_2/{\rm Dom}\dot{\Delta}):=(\dot{\Delta}^* u_1,u_2)_{L_2(X;L)}-(u_1,\dot{\Delta}^* u_2)_{L_2(X;L)}=[c_{u_1}h\overline{C_{u_2}}-C_{u_1}h\overline{c_{u_2}}](P).
\end{equation}
defines the complex symplectic (i.e. sesquilinear, skew-Hermitian, and non-degenerate) form on the quotient space ${\rm Dom}\dot{\Delta}^*/{\rm Dom}\dot{\Delta}$. 

Recall that $\mathscr{L}\subset {\rm Dom}\dot{\Delta}^*$ is the domain of some self-adjoint extension of $\dot{\Delta}$ if and only if $\mathscr{L}/{\rm Dom}\dot{\Delta}$ is a Lagrangian subspace of ${\rm Dom}\dot{\Delta}^*/{\rm Dom}\dot{\Delta}$. In view of (\ref{symp form expl}), $\{G(\cdot,y)/{\rm Dom}\dot{\Delta}, 1(\cdot,y)/{\rm Dom}\dot{\Delta}\}$ is the Darboux basis in ${\rm Dom}\dot{\Delta}^*/{\rm Dom}\dot{\Delta}$. Thus, all the Lagrangian subspaces in ${\rm Dom}\dot{\Delta}^*/{\rm Dom}\dot{\Delta}$ are given by
$$\mathcal{L}_\alpha=\{u/{\rm Dom}\dot{\Delta} \ | \ \big[c_u/C_u\big](P)={\rm ctg}\alpha\}$$
with $\alpha\in(-\pi/2,\pi/2]$. Therefore, all the self-adjoint extensions of $\dot{\Delta}$ are given by (\ref{Domains act}),(\ref{Domains}). From (\ref{Domains}) it easily follows that $\Delta_0=\Delta$. 

Introduce the sesquilinear form 
$$a(u,v):=(\dot{\Delta}u,v)_{L_2(X;L)}=(\overline{\partial}u,\overline{\partial}v)_{L_2(X;L\otimes\overline{K})}=\int_{X}\overline{\partial}u  h\rho^2\partial\overline{v}dS \qquad (u,v\in {\rm Dom}\dot{\Delta}).$$
This form admits the closure, also denoted by $a$. It is well-known  (see, e.g., \cite{BS}, Theorem 10.3.1) that the Friedrichs extension is the unique extension of $\dot{\Delta}$ whose domain is contained in ${\rm Dom}\,a$.

Note that the convergence in $H^1(X;L)$ implies the convergence in $a$-norm $\|u\|_a=(a(u,u))^{1/2}$. Using the same arguments as in Lemma \ref{H2 space lemma}, one can prove that $C_c^{\infty}(\dot{X};L)$ is dense in $H^1(X;L)$. Thus, $H^1(X;L)$ belongs to ${\rm Dom}\,a$. In particular, ${\rm Dom}\,\Delta\equiv H^2(X;L)\subset{\rm Dom}\,a$ and, thus, $\Delta$ is Friedrichs.  
\end{proof}
In the rest of the section, we compare parametrizations (\ref{Domains act}), (\ref{Domains}) and (\ref{domain via asymptotics act}), (\ref{domain via asymptotics}) of pseudo-laplacians.
\begin{lemma}
\label{parametrization comparing lemma}
Formula {\rm (\ref{parametrization comparing})} is valid.
\end{lemma}
\begin{proof}
Let $x$ be a holomorphic coordinate in a neighborhood of $P$ and $y=x(P)$. Let $\chi$ be a cut-off function equal to $1$ near $P$ and let the support of $\chi$ be sufficiently small. Introduce the section $G_{loc}$ of $L$ vanishing outside ${\rm supp}\chi$ by the equation $G_{loc}(x)=\chi(x)G_{as}$, where
$$G_{as}(x)=-\frac{1}{2\pi}{\rm log}|x-y|+\frac{\partial_y{\rm log} h(y)}{4\pi}(x-y){\rm log}(\overline{x-y})$$
in the local coordinate $x$. In view of (\ref{Dol}), we have
\begin{align*}
\Delta G_{loc}(x)&=-4\rho^2\partial\overline{\partial} G_{loc}(x)-4\rho^2(x)\partial_x{\rm log}h(x)\partial_{\overline{x}}G_{loc}(x)=[\Delta,\chi]G_{as}(x)+\\
+\chi(x)&\rho^2(x)\frac{\partial_x{\rm log}h(x)-\partial_y{\rm log} h(y)-\partial_x{\rm log}h(x)\partial_y{\rm log} h(y)(x-y)}{\pi(\overline{x-y})}=O(1), \qquad x\to y, \ x\ne y.
\end{align*}
Therefore, $\Delta[G_{loc}-G(\cdot,y)]\in L_2(X;L)$ in the sense of distributions. Due to the equivalence of norms (\ref{Greph norm}) and (\ref{Sobolev norm}), we have $G_{loc}-G(\cdot,y)\in H^2(X;L)$. Now, formulas (\ref{Green asymp}) and (\ref{domain of delta dot}) imply
$$G_{loc}=G(\cdot,y)-\mathsf{1}(\cdot,y)\big[m(P)+\rho(y)/2\pi]+\tilde{G}_{loc}, \qquad \tilde{G}_{loc}\in{\rm Dom}\,\dot{\Delta}.$$
Then any section $u$ given by (\ref{domain via asymptotics}) can be represented as
\begin{align*}
u=G(\cdot,y){\rm sin}\beta+\mathsf{1}(\cdot,y)[{\rm cos}\beta-(m(P)+\rho(y)/2\pi){\rm sin}\beta]+\tilde{u}, \qquad \tilde{u}\in{\rm Dom}\,\dot{\Delta}. 
\end{align*}
In particular, $u\in {\rm Dom}\Delta_{\alpha}$ where $\alpha$ is related to $\beta$ via (\ref{parametrization comparing}).
\end{proof}

\section{Comparison formulas for determinants}
\label{Comp F sec}
\paragraph{Comparison formula for the resolvents of $\Delta$ and $\Delta_\alpha$.} As mentioned in the introduction, we assume that ${\rm Ker}\,\Delta=\{0\}$. Suppose that $\lambda\in\mathbb{C}$ is not an eigenvalue of $\Delta$ and $\alpha\in(-\pi/2,0)\cup(0,\pi/2)$. Let $(\Delta-\lambda)u=f$. We search for the solution $u_\alpha$ to $(\Delta_\alpha-\lambda)u_\alpha=f$ of the form 
\begin{equation}
\label{solution alpha form}
u_\alpha=u+d(y)h(y)R_\lambda(\cdot,y),
\end{equation}
where $d\in\Gamma(X;L)$, $y=y(P)$ is a holomorphic coordinate of $P$, and $(x,y)\mapsto R_\lambda(x,y)$ is the resolvent kernel of $\Delta$. Since $(\Delta-\lambda)R_\lambda(\cdot,y)=0$ outside $P$, we have $(\dot{\Delta}^*-\lambda)u_\alpha=f$. In view of Hilbert's identity $R_\lambda(\cdot,y)-G(\cdot,y)=\lambda(\Delta-\lambda)^{-1}G(\cdot,y)$, we obtain
\begin{equation}
\label{special solution asymp}
h(y)R_\lambda(\cdot,y)=h(y)G(\cdot,y)+T(\lambda)\mathsf{1}(\cdot,y)+\tilde{R}_\lambda(\cdot,y),
\end{equation}
where $\tilde{R}_\lambda(\cdot,y)\in {\rm Dom}\,\dot{\Delta}$ and the number $T(\lambda)$ is called the scattering coefficient. Note that $T(0)=0$. As a corollary of (\ref{special solution asymp}), we have
$$u_\alpha=d(y)h(y)G(\cdot,y)+[u(y)+d(y)T(\lambda)]\mathsf{1}(\cdot,y)+\tilde{u}_\alpha,$$
where $\tilde{u}_\alpha\in {\rm Dom}\,\dot{\Delta}$. Comparing the last formula with (\ref{Domains}), we conclude that $u_\alpha\in{\rm Dom}\,\Delta_\alpha$ if and only if
\begin{equation}
\label{coefficient d}
d(y)=\frac{u(y)}{{\rm ctg}\alpha-T(\lambda)}=\frac{(f,R_{\overline{\lambda}}(\cdot,y))_{L_2(X;L)}}{{\rm ctg}\alpha-T(\lambda)}.
\end{equation}
Since $R_\lambda$ is the resolvent kernel of $\Delta$, we have $u(y)=(f,\overline{R_{\lambda}(y,\cdot)})_{L_2(X;L)}=(f,R_{\overline{\lambda}}(\cdot,y))_{L_2(X;L)}$. Therefore, formulas (\ref{solution alpha form}) and (\ref{coefficient d}) imply 
\begin{equation}
\label{resolvent difference}
[(\Delta_\alpha-\lambda)^{-1}-(\Delta-\lambda)^{-1}]f=u_\alpha-u=\frac{(f,R_{\overline{\lambda}}(\cdot,y))_{L_2(X;L)}h(y)R_\lambda(\cdot,y)}{{\rm ctg}\alpha-T(\lambda)}
\end{equation}
(here the denominator in the right-hand side equals zero if and only if $\lambda$ is an eigenvalue of $\Delta_{\alpha}$). 

Note that, in the right-hand side of (\ref{resolvent difference}), the one-dimensional operator acts on $f$. Then
\begin{equation}
\label{trace of the difference}
{\rm Tr}[(\Delta_\alpha-\lambda)^{-1}-(\Delta-\lambda)^{-1}]=\frac{h(y)(R_\lambda(\cdot,y),R_{\overline{\lambda}}(\cdot,y))_{L_2(X,L)}}{{\rm ctg}\alpha-T(\lambda)}.
\end{equation}
Since $(\Delta_\alpha-i)^{-1}-(\Delta-i)^{-1}$ is a one-dimensional operator, the essential spectra of $\Delta_\alpha$ and $\Delta$ coincide (see Theorem 9.1.4, \cite{BS}). Since the spectrum of $\Delta$ is discrete, the spectrum of any $\Delta_\alpha$ is also discrete. Also, since $\Delta$ is the Friedrichs extension of $\dot{\Delta}$, we have $\Delta_\alpha<\Delta$ for $\alpha\in(0,\pi)$ (see Corollary 10.3.2, \cite{BS}) and, since the spectra of the operators $\Delta_\alpha$ and $\Delta$ are discrete, their exact lower bounds obey $m_{\Delta}>m_{\Delta_\alpha}$. In view of Theorems 10.3.7 and 10.3.8, \cite{BS}, there is exactly one eigenvalue $\lambda_1(\Delta_\alpha)$ which does not belong to $[m_{\Delta},+\infty)$. In particular, each $\Delta_\alpha$ is semi-bounded.

Differentiating the equation $(\Delta-\lambda)R_\lambda(\cdot,y)=0$ in $\dot{X}$, one obtains $(\Delta-\lambda)\partial_\lambda R_\lambda(\cdot,y)=R_\lambda(\cdot,y)$ while (\ref{special solution asymp}) implies
$$h(y)\partial_\lambda R_\lambda(\cdot,y)=\partial_{\lambda}T(\lambda)\mathsf{1}(\cdot,y)+W(\cdot,y),$$
where $W(\cdot,y)\in {\rm Dom}\,\dot{\Delta}$. Hence
$$\partial_{\lambda}T(\lambda)=h(y)\partial_\lambda R_\lambda(y,y)=h(y)(R_\lambda(\cdot,y),R_{\overline{\lambda}}(\cdot,y))_{L_2(X;L)}.$$
Now (\ref{trace of the difference}) takes the form
\begin{equation}
\label{trace of the difference 1}
{\rm Tr}[(\Delta_\alpha-\lambda)^{-1}-(\Delta-\lambda)^{-1}]=-\partial_\lambda{\rm log}\big({\rm ctg}\alpha-T(\lambda)\big).
\end{equation}

\paragraph{Comparison formula for $\zeta(s|\Delta)$ and $\zeta(s|\Delta_\alpha)$.} Suppose that ${\rm Ker}\,\Delta_\alpha=\{0\}$.  We define $\lambda^{-s}:={\rm exp}(-s{\rm log}\lambda)$, where the cut for the logarithm is a simple path $\varpi_{cut}$ going from $\lambda=-\infty$ to $\lambda=0$ which does not contain eigenvalues of $\Delta$ and $\Delta_\alpha$. We assume that $\varpi_{cut}$ coincides with the semi-axis $(-\infty,a_0]$ outside the semi-plane $\Re\lambda>a_0$ (where $a_0<\min\{m_{\Delta},m_{\Delta_\alpha}\}$) and with the semi-axis $\lambda<0$ in a small neighborhood of $\lambda=0$. For $\Re s>0$ and $A=\Delta$ or $A=\Delta_\alpha$, we have
\begin{equation}
\label{powers of operator via resolvent}
A^{-s}=\frac{1}{2\pi i}\int\limits_\varpi(A-\mu)^{-1}\mu^{-s}d\mu \qquad (\varpi:=\partial\big(\mathbb{C}\backslash( \gamma_{cut}\cup U_\epsilon)\big)).
\end{equation}
where $\varpi$ is the boundary of the domain obtained from $\mathbb{C}$ by deleting $\varpi_{cut}$ and a small $\epsilon$-neighborhood $U_\epsilon$ of $\mu=0$. Since the difference $(\Delta_\alpha-\lambda)^{-1}-(\Delta-\lambda)^{-1}$ is a one-dimensional operator for any $\lambda$, the integrals of it converge in both operator and trace norms. Then (\ref{powers of operator via resolvent}) and (\ref{trace of the difference 1}) imply
\begin{equation}
\label{difference of zetas via resolvent}
\begin{split}
\zeta(s|\Delta)-\zeta(s|\Delta_\alpha)=\int_\varpi \partial_\mu{\rm log}\big({\rm ctg}\alpha-T(\mu)\big)\frac{\mu^{-s}d\mu}{2\pi i}&=\\
=sJ_0(s)+\pi^{-1}{\rm sin}(\pi s)\big[e^{-\pi is}J_{-\infty}(s)-&{\rm log}\big({\rm ctg}\alpha-T(-\epsilon)\big)\epsilon^{-s}\big].
\end{split}
\end{equation}
where 
$$J_0(s)=\int_{|\mu|=\epsilon} {\rm log}\big({\rm ctg}\alpha-T(\mu)\big)\frac{\mu^{-(s+1)}d\mu}{2\pi i}$$
is an entire function of $s$ and 
$$J_{-\infty}(s)=\int_{\varpi_{cut}\backslash U_\epsilon} \partial_\mu{\rm log}\big({\rm ctg}\alpha-T(\mu)\big)\mu^{-s}d\mu.$$
To study the analyticity properties of $J_{-\infty}$, we derive the asymptotics of $T(\lambda)$ as $\lambda\to -\infty$. To this end, let us recall the following asymptotics of the resolvent kernel (see formulas (2.32) on p.38 and (2.25) on p.34, \cite{Fay}) 
\begin{equation}
\label{Fays asymptotics}
h(y)R_{\lambda}(x,y)+\frac{1}{2\pi}d(x,y)=\frac{1}{4\pi}\Big[a_0-{\rm log}(|\lambda|+1)+\frac{a_{-1}(y)}{(|\lambda|+1)}\Big]+\tilde{R}_{\lambda}(x,y) \ (x\to y).
\end{equation}
Here $y=x(P)$ and the remainder $\tilde{R}_{\lambda}(x,y)$ is continuous at $x=y$ and obeys the (admitting differentiation) estimate $\tilde{\mathcal{R}}_{\lambda}(y,y|\Delta)=O(\lambda^{-2})$. The coefficients in (\ref{Fays asymptotics}) are given by
$$a_0=2({\rm log}2-\gamma), \quad a_{-1}=1+R+K/3,$$
where $K=4\rho^2\partial\overline{\partial}{\rm log}\rho$ and $R=-2\rho^2\partial\overline{\partial} {\rm log}h$ are the scalar curvatures of the metrics $\rho^{-2}|dz|^2$ and $h$. Comparing formulas (\ref{Fays asymptotics}), (\ref{Green asymp}) and (\ref{special solution asymp}), we obtain 
\begin{align*}
T(\lambda)=h(y)[R_\lambda-G](y,y)=\frac{1}{4\pi}\Big[a_0-{\rm log}(|\lambda|+1)+\frac{a_{-1}(y)}{(|\lambda|+1)}\Big]-m(P)+O(\lambda^{-2}).
\end{align*}
Therefore,
\begin{align*}
\partial_\lambda{\rm log}\big({\rm ctg}\alpha-T(\lambda)\big)=\frac{\partial_\lambda T(\lambda)}{T(\lambda)-{\rm ctg}\alpha}=\frac{-1}{|\lambda|\big(\mathfrak{q}+{\rm log}|\lambda|\big)}+\tilde{q}(\lambda),
\end{align*}
where $\mathfrak{q}=4\pi[m(P)+{\rm ctg}\alpha]-a_0$ and $\tilde{q}(\lambda)=O(|\lambda|^{-2}) \quad (\lambda\to-\infty)$. Thus,
\begin{align*}
-e^{-\pi is}J_{-\infty}(s)=e^{-\pi is}\int_{\varpi_{cut}\backslash U_\epsilon} \frac{\mu^{-s}d\mu}{|\mu|\big(\mathfrak{q}+{\rm log}|\mu|\big)}+\tilde{J}_{-\infty}(s).
\end{align*}
The remainder
$$\tilde{J}_{-\infty}(s)=-e^{-\pi is}\int_{\varpi_{cut}\backslash U_\epsilon}\mu^{-s}\tilde{q}(\mu)d\mu$$
is analytic for $\Re s>-1$. In the last two formulas, one can replace the integration contour in the right-hand sides by $(-\infty,-\epsilon)$ (then $\mu^{-s}d\mu=-|\mu|^{-s}e^{\pi i s}d|\mu|$). Thus,
\begin{align*}
-e^{-\pi is}J_{-\infty}(s)-\tilde{J}_{-\infty}(s)&=\int_{\epsilon}^{+\infty}\frac{t^{-(s+1)}dt}{{\rm log}t+\mathfrak{q}}=\\
=&e^{s\mathfrak{q}}\int_{s({\rm log}\epsilon+\mathfrak{q})}^{+\infty}\frac{e^{-p}dp}{p}=-e^{s\mathfrak{q}}{\rm Ei}(-s({\rm log}\epsilon+\mathfrak{q})),
\end{align*}
where $p=s({\rm log}t+\mathfrak{q})$ and ${\rm Ei}$ denotes the exponential integral (cf. \cite{KLP}). Now (\ref{difference of zetas via resolvent}) takes the form
\begin{align*}
\zeta(s|\Delta)-\zeta(s|\Delta_\alpha)=sJ_{0}(s)&+\\
+\pi^{-1}{\rm sin}(\pi s)\Big[e^{s\mathfrak{q}}{\rm Ei}&(-s({\rm log}\epsilon+\mathfrak{q}))-\tilde{J}_{-\infty}(s)-{\rm log}\big({\rm ctg}\alpha-T(-\epsilon)\big)\epsilon^{-s}\Big].
\end{align*}
In view of the series representation
$${\rm Ei}(z)={\rm log}z+\gamma+z+O(z^2) \qquad (z\to 0, \ {\rm arg}z\in[-\pi,\pi)),$$
we have
\begin{align*}
[\zeta(s|\Delta_\alpha)+s{\rm log}s]-\zeta(s|\Delta)&=\\
=-s\Big[{\rm log}(-({\rm log}\epsilon+&\mathfrak{q}))+\gamma+\int_{-\infty}^{-\epsilon}\tilde{q}(\mu)d\mu+{\rm log}\frac{{\rm ctg}\alpha-T(0)}{{\rm ctg}\alpha-T(-\epsilon)}\Big]+\tilde{o}_2(s),
\end{align*}
where $s\mapsto\tilde{o}_2(s)$ is analytic near $s=0$ and $s=0$ is a zero of $\tilde{o}_2$ of order $2$. Thus, $s\mapsto\zeta(s|\Delta_\alpha)$ has logarithmic singularity at $s=0$ and one needs to apply regularization (\ref{reg zeta}). Then the regularized zeta function $s\mapsto\zeta^{(r)}(s|\Delta_\alpha)$ is analytic near $s=0$, and
\begin{align}
\label{difference of zetas via resolvent 2}
\begin{split}
-\partial_s[\zeta^{(r)}(s|\Delta_\alpha)-\zeta(s|\Delta)]\big|_{s=0}={\rm log}(-({\rm log}\epsilon&+\mathfrak{q}))+\gamma+
\\+&\int_{-\infty}^{-\epsilon}\tilde{q}(\mu)d\mu+{\rm log}\frac{{\rm ctg}\alpha-T(0)}{{\rm ctg}\alpha-T(-\epsilon)}
\end{split}
\end{align}
for sufficiently small $\epsilon>0$. Note that the left-hand side of (\ref{difference of zetas via resolvent 2}) is independent of $\epsilon$ while the right-hand side is real-analytic in $\epsilon\in (0,+\infty)$. Then the right-hand side is independent of $\epsilon\in (0,+\infty)$. Sending $\epsilon$ to infinity and taking into account that $T(0)=0$, we arrive at
\begin{align}
\label{difference zeta}
-\partial_s[\zeta^{(r)}(s|\Delta_\alpha)-\zeta(s|\Delta)]\big|_{s=0}={\rm log}(-4\pi{\rm ctg}\alpha)+\gamma.
\end{align}
Comparison formula (\ref{comparison formula determinants}) follows from (\ref{difference zeta}) and definition (\ref{reg zeta}) of the regularized determinant ${\rm det}^{(r)}\Delta_\alpha$. Formula (\ref{comparison formula determinants coordepend}) follows from (\ref{comparison formula determinants}) and Lemma \ref{parametrization comparing lemma}.

\section{Explicit formulas for Robin mass}
\label{sec Robin}
\subsection{Derivation of formula (\ref{Robin mass explicit})}
Choose a canonical basis $\{a_i,b_j\}_{i,j=1}^g$ of cycles; let $\vec{v}=(v_1,\dots,v_g)^t$ be the basis of Abelian differentials on $X$ normalized with respect to $\{a_i,b_j\}_{i,j=1}^g$, and let $\mathbb{B}$ be the matrix of $b$-periods of $X$ (see, e.g., \cite{GH}, p. 231). Denote by $\mathcal{A}(\mathcal{D})$ the Abel transform of the divisor $\mathcal{D}$ with the basepoint $Q$; then $\mathcal{A}(y-x)=\int_{x}^y \vec{v}$. Let $\mathcal{K}$ denote the vector of Riemann constants, associated with the same basepoint $Q$.

From now on, we assume that $L$ obeys (\ref{bundle restrictions}). Then $L\simeq{\bm\triangle}\otimes\chi$, where ${\bm\triangle}$ is the `basic' spinor bundle obeying $\mathcal{A}({\bm\triangle})=-\mathcal{K}$ while $\chi$ is a unitary holomorphic line bundle (see Example 2.3 on pp.28,29, \cite{Fay}).

The {\it Szeg\"o kernel} $S$ is defined as a section of $L\hat{\otimes}KL^{-1}$ given by
\begin{equation}
\label{Szego Fay}
S(x,y)=-4\pi h(y)\partial_{y}G(x,y)
\end{equation}
(see p.25, \cite{Fay}). The reversal of (\ref{Szego Fay}) is 
\begin{equation}
\label{Szego inverse Fay}
G(x,y)=\frac{1}{4\pi^2}\int_{X}\mathcal{S}(x,z)h^{-1}(z)\overline{\mathcal{S}(y,z)}\hat{d}z
\end{equation}
(see (2.6), \cite{Fay}), where $\hat{d}z=d\overline{z}\wedge dz/2i$. In view of conditions (\ref{bundle restrictions}), the Szeg\"o kernel is independent of the choice of metrics and coincides with integral kernel of the operator $-\pi\overline{\partial}^{-1}$. Moreover, it is biholomorphic outside the diagonal $x=y$ and obeys the asymptotics 
\begin{equation}
\label{Szego Fay near diagonal}
S(x,y)=\frac{1}{y-x}+O(1) \qquad (|x-y|\to 0)
\end{equation}
(see p. 25-29, \cite{Fay}). In addition, the following explicit formula for the Szeg\"o kernel holds
\begin{equation}
\label{Szego explicit}
\mathcal{S}(x,y)=\frac{\theta[\chi](\mathcal{A}(y-x))}{\theta[\chi](0)E(x,y)},
\end{equation}
where $E(x,y)$ is the prime-form of $X$ and $\theta[\chi](\cdot)$ is the theta-function (defined in \cite{Fay}, (1.9)).

Formulas (\ref{Szego inverse Fay}) and (\ref{Szego explicit}) provide an explicit expression for the Green function $G$. To obtain explicit formula (\ref{Robin mass explicit}) for $m(y)$, one needs a regularization of the (diverging at $x=y$) integral in the right-hand side of (\ref{Szego inverse Fay}). To this end, let us introduce the symmetric real-valued function
\begin{equation}
\label{Phi func}
(x,y)\mapsto\Phi(x,y)=-\frac{1}{4\pi}{\rm log}\left[\frac{F(x,y)}{\rho(x)\rho(y)}\right],
\end{equation}
on $X\times X$, where $F$ is given by (\ref{F section}). Due to the asymptotics (see \cite{Fay}, (1.3))
$$\frac{E(x,y)}{x-y}=1+O(|x-y|^2),$$
formulas (\ref{Phi func}) and (\ref{F section}) imply
\begin{equation}
\label{asymp of Phi}
\Phi(x,y)=-\frac{1}{2\pi}{\rm log}d(x,y)+O(|x-y|), \quad 4\pi\partial_{y}\Phi(x,y)=\frac{1}{x-y}+O(1) \qquad (|x-y|\to 0).
\end{equation}
Then
\begin{equation}
\label{reg 1}
m(y)=\lim_{x\to y}\big(h(y)G(x,y)-\Phi(x,y)\big).
\end{equation}

Let $x\ne y$ and let $X_\epsilon(x,y)$ be the domain obtained by removing $\epsilon$-neighborhoods (in the metric $\rho^{-2}|dz|^2$) of $x$ and $y$. In view of the Stokes theorem and (\ref{asymp of Phi}), we have
\begin{align}
\label{phi integral 1}
\begin{split}
\int\limits_{X_\epsilon(x,y)}4[\partial_{\overline{z}}\partial_z\Phi(x,z)\,\overline{\Phi(z,y)}+\partial_z\Phi(x,z)\,\overline{\partial_z\Phi(z,y)}]\hat{d}z&=\\
=\int\limits_{\partial X_\epsilon(x,y)}\Big[\frac{\overline{\Phi(x,y)}}{x-z}+O(1)+O\big(\big|{\rm log}|z-y|\big|\big)&\Big]\frac{dz}{2\pi i}=\Phi(x,y)+o(1).
\end{split}
\end{align}
Since the prime-form $E$ is biholomorphic, we have $\partial_{\overline{z}}\partial_z{\rm log}|E(x,z)|^2=0$ ($x\ne z$). Then formulas (\ref{Phi func}) and (\ref{F section}) imply
\begin{align}
\label{Laplace of Phi}
4\partial_{\overline{z}}\partial_z\Phi(x,z)=\partial_{\overline{z}}\partial_z\Big[\frac{1}{\pi}{\rm log}\rho(z)-\frac{1}{2}(\vec{A}-\overline{\vec{A}})^t (\Im\mathbb{B})^{-1}(\vec{A}-\overline{\vec{A}})\Big]=\frac{K(z)}{4\pi\rho^2(z)}+\overline{\vec{v}(z)}^t(\Im\mathbb{B})^{-1}\vec{v}(z)
\end{align}
for $z\ne x$, where $$\vec{A}=\int_z^x \vec{v}, \quad \partial_{\overline{z}}\vec{A}=0, \quad \partial_{z}\vec{A}=-\vec{v}(z)$$
and $K=4\rho^2\partial\overline{\partial}{\rm log}\rho$ is the Gaussian curvature of the metric $\rho^{-2}|dz|^2$. Now passing to the limit $\epsilon\to 0$ in (\ref{phi integral 1}) yields
\begin{align}
\label{phi integral}
\begin{split}
\Phi(x,y)=\int\limits_{X}4\partial_z\Phi(x,z)\,\overline{\partial_z\Phi(z,y)}\hat{d}z+\int\limits_{X}\Big[\frac{K(z)}{4\pi\rho^2(z)}+\overline{\vec{v}(z)}^t(\Im\mathbb{B})^{-1}\vec{v}(z)\Big]\overline{\Phi(z,y)}\hat{d}z.
\end{split}
\end{align}
Substituting (\ref{Szego inverse Fay}) and (\ref{phi integral}) into (\ref{reg 1}), one obtains
\begin{align}
\label{reg 2}
\begin{split}
m(y)=\lim_{x\to y}\left[\frac{1}{4\pi^2}\int_{X}\Big[\mathcal{S}(x,z)\overline{\mathcal{S}(y,z)}\frac{h(y)}{h(z)}-16\pi^2\partial_z\Phi(x,z)\,\overline{\partial_z\Phi(z,y)}\Big]\hat{d}z\right]-\\
-\int\limits_{X}\left[\frac{K(z)}{4\pi\rho^2(z)}+\overline{\vec{v}(z)}^t(\Im\mathbb{B})^{-1}\vec{v}(z)\right]\overline{\Phi(z,y)}\hat{d}z.
\end{split}
\end{align}
In view of asymptotics (\ref{Szego Fay near diagonal}) and (\ref{asymp of Phi}), the section
$$(y,z)\mapsto |\mathcal{S}(y,z)|^2\frac{h(y)}{h(z)}-16\pi^2|\partial_z\Phi(y,z)|^2$$
of $1\hat{\otimes}K\overline{K}$ is integrable in $z\in X$. Therefore, one can interchange passing to the limit and the integration in (\ref{reg 2}). As a result, one arrives at
\begin{align}
\label{m regularization}
\begin{split}
m(y)=\frac{1}{4\pi^2}\int\limits_{X}\Big[|\mathcal{S}(y,z)|^2&\frac{h(y)}{h(z)}-16\pi^2|\partial_z\Phi(y,z)|^2\Big]\hat{d}z-\\
-&\int\limits_{X}\left[\frac{K(z)}{4\pi\rho^2(z)}+\overline{\vec{v}(z)}^t(\Im\mathbb{B})^{-1}\vec{v}(z)\right]\overline{\Phi(z,y)}\hat{d}z.
\end{split}
\end{align}
To derive (\ref{Robin mass explicit}), it remains to substitute (\ref{Szego explicit}), (\ref{Phi func}) and (\ref{F section}) into (\ref{m regularization}).

\subsection{Relation between the Robin masses for conformally equivalent metrics} 
Let $\rho'^{-2}|dz|^2$ and $h'$ and $\rho^{-2}|dz|^2$ and $h$ be two pairs of metrics on the Riemann surface $X$ and the holmorphic line bundle $L$, respectively. Denote by $G'$ and $m'$ the Green function and the Robin mass for the Laplacian $\Delta'$ associated with the surface $(X,\rho'^{-2})$ and the hermitian bundle $(L,h')$. 

Suppose that $L$ satisfies (\ref{bundle restrictions}). Then Szeg\"o kernel (\ref{Szego Fay}) is independent of the choice of conformal metrics and formulas (\ref{Szego Fay}) and (\ref{Szego inverse Fay}) remain valid after replacing $G,h$ by $G',h'$. Then
\begin{align*}
G'(x,y)=\frac{1}{4\pi^2}\int_{X}\mathcal{S}(x,z)h'^{-1}(z)\overline{\mathcal{S}(y,z)}\hat{d}z=\frac{-1}{\pi}\int_{X}\frac{h}{h'}(z)\partial_{z}G(x,z)\overline{\mathcal{S}(y,z)}\hat{d}z.
\end{align*}
Since $\mathcal{S}(y,z)$ is biholomorphic outside $y=z$, we have 
\begin{align*}
G'(x,y)=\frac{1}{\pi}\int_{X}\Big[\partial_{z}\frac{h}{h'}\Big](z)G(x,z)\overline{\mathcal{S}(y,z)}\hat{d}z-\frac{1}{\pi}\int_{X}\partial_{z}\Big[\frac{h}{h'}(z)G(x,z)\overline{\mathcal{S}(y,z)}\Big]\hat{d}z.
\end{align*}
In view of the Stokes theorem and asymptotics (\ref{Szego Fay near diagonal}) and (\ref{Green asymp}), the last integral in the right-hand side is equal to $\pi h(y)h'^{-1}(y)G(x,y)$. Thus,
\begin{equation}
\label{comparing green f conf m 1}
h'(y)G'(x,y)-h(y)G(x,y)=\frac{h'(y)}{\pi}\int_{X}\Big[\partial_{z}\frac{h}{h'}\Big](z)G(x,z)\overline{\mathcal{S}(y,z)}\hat{d}z.
\end{equation}
In view of (\ref{Green asymp}), we have
\begin{align*}
h'(y)G'(x,y)&=-\frac{1}{2\pi}{\rm log}\,\big[|x-y|\rho'^{-1}(y)\big]+m'(y)+o(1),\\
h(y)G(x,y)&=-\frac{1}{2\pi}{\rm log}\,\big[|x-y|\rho^{-1}(y)\big]+ m(y)+o(1)
\end{align*}
as $x\to y$. Then passing to the limit as $x\to y$ in (\ref{comparing green f conf m 1}) yields the comparison formula
\begin{equation}
\label{comparing green f conf m}
\begin{split}
m'(y)-m(y)=&\frac{1}{2\pi}{\rm log}\,\Big[\frac{\rho(y)}{\rho'(y)}\Big]+\frac{h'(y)}{\pi}\int_{X}\Big[\partial_{z}\frac{h}{h'}\Big](z)G(y,z)\overline{\mathcal{S}(y,z)}\hat{d}z=\\
=&\frac{1}{2\pi}{\rm log}\,\Big[\frac{\rho(y)}{\rho'(y)}\Big]-4h'(y)\int_{X}\Big[\partial_{z}\frac{h}{h'}\Big](z)G(y,z)h(z)\partial_{\overline{z}}G(z,y)\hat{d}z
\end{split}
\end{equation}
(cf. p.203, \cite{Nev}).

\subsection{Examples}
\paragraph{The Robin mass for the spinor bundle on the round sphere.} Let $x$ and $x'=1/x$ be the system of holomorphic coordinates on the Riemann sphere $\overline{\mathbb{C}}$ and $L=C=\sqrt{K}$ be the (unique up to isomorphism) spinor bundle on $\overline{\mathbb{C}}$. Then its Szeg\"o kernel is given by 
$S(x,y)=(y-x)^{-1}\sqrt{dxdy}$. Note that the prime-form on $\overline{C}$ is just $E(x,y)=(x-y)/\sqrt{dxdy}$. 

The round metric $\rho^{-2}|dx|^2$ on $\overline{\mathbb{C}}$ is given by $\rho(x)=1+|x|^2$; then its Gaussian curvature is constant $K=4$. The metric in the spinor bundle $C$ is given by $h=\rho$. The Green function $G$ of the spinor Laplacian $\Delta$ on the sphere $\overline{\mathbb{C}}$ is invariant with respect to rotations. Therefore, the Robin mass $m$ is constant on $\overline{\mathbb{C}}$. 

In contrast to (\ref{Robin mass explicit}), formula (\ref{m regularization}) is still valid for the case $g=0$ and it takes the form
\begin{align*}
4\pi^2 m=&\int\limits_{\mathbb{C}}\Big[|\mathcal{S}(0,z)|^2\frac{h(0)}{h(z)}-16\pi^2|\partial_z\Phi(0,z)|^2-\frac{\pi K}{\rho^2(z)}\overline{\Phi(z,0)}\Big]\hat{d}z=\\
=&\left[\Phi(x,y)=\frac{-1}{4\pi}{\rm log}\left[\frac{|x-y|^2}{\rho(x)\rho(y)}\right], \ |z|^2=t\right]=\pi\int\limits_{0}^{+\infty}\frac{1+{\rm log}[t/(1+t)]}{(1+t)^2}dt=0
\end{align*}
Comparison of the last formula with the explicit expression
$$h(0)G(x,0)=G(x,0)=\frac{1}{4\pi}{\rm log}[1+|x|^{-2}]=-\frac{1}{2\pi}{\rm log}|x|+O(|x|^2)$$
for the spinor Green function $G$ provides a simple cross-check of (\ref{m regularization}).

\paragraph{The Robin masses for spinor bundles on flat tori.} Let $\mathbb{T}$ be the torus $\mathbb{C}/(\mathbb{Z}+\tau\mathbb{Z})$ with $\Im\tau>0$. Let $z\in\mathbb{C}$ be a coordinate of the point $z/(\mathbb{Z}+\tau\mathbb{Z})$ of $\mathbb{T}$. The metric on $\mathbb{T}$ is $|dz|^2$; then the area of $\mathbb{T}$ is $\Im\tau$. 

The sections $f$ of any line bundle $L$ over $\mathbb{T}$ can be considered as a functions on the universal cover $\mathbb{C}$ of $\mathbb{T}$ obeying the quasi-periodicity conditions 
\begin{equation}
\label{quasi periodicity}
f(z+1)=\mathfrak{s}_1(z) f(z), \qquad f(z+\tau)=\mathfrak{s}_\tau(x) f(z),
\end{equation}
where the automorphy factors $\mathfrak{s}_1,\mathfrak{s}_\tau$ are invariant under the cover transformations $\mathbb{Z}+\tau\mathbb{Z}$. There are 4 non-isomorphic spinor bundles $C_{\mathfrak{s}_1,\mathfrak{s}_\tau}$ where $\mathfrak{s}_1,\mathfrak{s}_\tau=\pm 1$. 

The metric of $C_{\mathfrak{s}_1,\mathfrak{s}_\tau}$ is given by $h=1$. The the spinor Laplacians are given by $\Delta=\partial_{z}\partial_{\overline{z}}$ in local coordinates. Note that the kernel of $\Delta$ is non-trivial only for $C=C_{+,+}=1$. The Greens functions for Laplacians on $C_{\mathfrak{s}_1,\mathfrak{s}_\tau}$ are invariant with respect to translations of torus: $G(x,y)=G(x-y)$. Then the the Robin masses corresponding to $C_{\mathfrak{s}_1,\mathfrak{s}_\tau}$ are constant on $\mathbb{T}$. The Green function for $C_{+,+}=1$ is given by 
\begin{equation*}
G(z|\tau)=-\frac{1}{2\pi}{\rm log}\left|\frac{\theta_1(z|\tau)}{\theta'_1(0|\tau)}\right|+\frac{(\Im z)^2}{2\Im\tau}.
\end{equation*}
In view of (\ref{quasi periodicity}), the Green function for $C_{+,-}, C_{-,+}, C_{-,-}$ are given by
\begin{align*}
G_{+,-}(z|\tau)&=G(z|2\tau)-G(z-\tau|2\tau)=\frac{1}{2\pi}{\rm log}\left|\frac{\theta_1(z-\tau|2\tau)}{\theta_1(z|2\tau)}\right|+\frac{\Im (2z-\tau)}{4},\\
G_{-,+}(z|\tau)&=G\Big(\frac{z}{2}\Big|\frac{\tau}{2}\Big)-G\Big(\frac{z-1}{2}\Big|\frac{\tau}{2}\Big)=\frac{1}{2\pi}{\rm log}\left|\frac{\theta_1(\frac{z-1}{2}|\frac{\tau}{2})}{\theta_1(\frac{z}{2}|\frac{\tau}{2})}\right|,\\
G_{-,-}(z|\tau)&=G\Big(\frac{z}{2}\Big|\tau\Big)-G\Big(\frac{z-1}{2}\Big|\tau\Big)-G\Big(\frac{z-\tau}{2}\Big|\tau\Big)+G\Big(\frac{z-1-\tau}{2}\Big|\tau\Big)=\\
&\qquad\qquad\qquad\qquad\qquad\qquad\qquad=\frac{1}{2\pi}{\rm log}\left|\frac{\theta_1(\frac{z-1}{2}|\tau)\theta_1(\frac{z-\tau}{2}|\tau)}{\theta_1(\frac{z}{2}|\tau)\theta_1(\frac{z-1-\tau}{2}|\tau)}\right|,
\end{align*}
respectively. Therefore,
\begin{align*}
m_{+,+}&=0, & m_{-,-}=\frac{1}{2\pi}{\rm log}\left|\frac{2\theta_1(\frac{1}{2}|\tau)\theta_1(\frac{\tau}{2}|\tau)}{\theta_1(0|\tau)\theta_1(\frac{1+\tau}{2}|\tau)}\right|,\\
m_{+,-}&=\frac{1}{2\pi}{\rm log}\left|\frac{\theta_1(\tau|2\tau)}{\theta'_1(0|2\tau)}\right|-\frac{\Im\tau}{4}, & m_{-,+}=\frac{1}{2\pi}{\rm log}\left|\frac{2\theta_1(\frac{1}{2}|\frac{\tau}{2})}{\theta'_1(0|\frac{\tau}{2})}\right|.
\end{align*}

\section{On Steiner's relation between regularized $\zeta(1|\Delta)$ and the Robin mass}
\label{Steiner sec}
For the case of the trivial bundle $L$, relation (\ref{ADM via Robin}) between regularized $\zeta(1|\Delta)$ (given by (\ref{reg zeta 1})) and the Robin mass is proved in Proposition 2, \cite{Steiner}. In this section, we provide a straightforward generalization of this result to the case of arbitrary $L$. For simplicity, we assume that ${\rm Ker}\,\Delta=\{0\}$ (if ${\rm Ker}\,\Delta\ne\{0\}$, the zero modes are excluded from the definition of $\zeta(s|\Delta)$ and $\mathrm{K}(x,y,t)$ in the formulas below should be replaced by $\mathrm{K}(x,y,t)-B(x,y)$, where $B$ is the Bergman kernel defined after (\ref{bergman})). 

Let $x,y,t\mapsto \mathrm{K}(x,y,t)$ be the heat kernel associated with the equation $(\partial_t+\Delta)u(x,t)=0$. According to Theorem 2.5 and formulas (2.24) and (2.25) on p.34, \cite{Fay}, $\mathrm{K}(x,y,t)$ admits the asymptotics
\begin{equation}
\label{heat asymp}
\mathrm{K}(x,y,t)h(y)=\frac{{\rm exp}\big(-r^{2}/4t\big)}{4\pi t}\big[1+\psi_0(x,y)\big]+\Psi_1(x,y,t),
\end{equation}
where $r=d(x,y)$ and
\begin{align*}
\psi_0(x,y)=[\partial_y{\rm log}h](y)(y-x)+&\big[(\partial_y{\rm log}h)^2-\partial^2_{y}h/2h\big](y-x)^2+\\
+\big[K(y)/3+&R(y)\big]|y-x|^2/4\rho^{2}(y)+O(|x-y|^3)=O(r),
\end{align*}
while the remainder $\Psi_1(x,y,t)$ is bounded uniformly in $x,y\in X$ and $t\ge 0$. Here $R=-2\rho^2\overline{\partial}\partial{\rm log}h$ is the scalar curvature of $h$.

The kernels $x,y\mapsto G^{(s)}(x,y)$ of the operators $\Delta^{-s}$ are related to the heat kernel via
\begin{align*}
h(y)G^{(s)}(x,y)=\frac{h(y)}{\Gamma(s)}\int_0^{+\infty}\mathrm{K}(x,y,t)t^{s-1}dt&=\\
=\frac{1+\psi_0(x,y)}{4\pi \Gamma(s)}&\int_0^{1}{\rm exp}\big(-r^2/4t\big)t^{s-2}dt+\mathcal{K}_1(s,x,y),
\end{align*}
where
\begin{align*}
\mathcal{K}_1(s,x,y)=\frac{h(y)}{\Gamma(s)}\Big(\int_1^{+\infty}\mathrm{K}(x,y,t)t^{s-1}dt+\int_0^{1}\Psi_1(x,y,t)t^{s-1}dt\Big).
\end{align*}
In view of (\ref{heat asymp}), $G^{(s)}(x,y)$ is well defined for any $x,y\in X$ for $\Re s>1$ and for any $s\in\mathbb{C}$ for $x\ne y$. Note that $\mathcal{K}_1(s,x,y)$ is bounded in $x,y\in X$ and analytic in $s$ near $s=1$. The integral $\int_0^{1}{\rm exp}\big(-r^2/4t\big)t^{s-2}dt$ is analytic with respect to $r$,$s$ and is well-defined for any $s\in\mathbb{C}$ and $\Re r^2>0$. Denote $u:=r^2/4t$. For $r>0$ and $1/2<\Re s<1$, we have
\begin{align}
\label{gjfgjkk}
\begin{split}
\int\limits_0^{1}{\rm exp}\big(-r^2/4t\big)t^{s-2}dt=(r^2/4)^{s-1}\Big(\int\limits_{0}^{+\infty}-\int\limits_{0}^{r^2/4}\Big)e^{-u}u^{-s}du=\\
=(r^2/4)^{s-1}\Big(\Gamma(1-s)-\int_{0}^{r^2/4}\big(e^{-u}-1\big)u^{-s}du-\int_{0}^{r^2/4}u^{-s}du\Big)=\\
=(r^2/4)^{s-1}\Gamma(1-s)-\frac{1}{1-s}-(r^2/4)^{s-1}\int_{0}^{r^2/4}\big(e^{-u}-1\big)u^{-s}du.
\end{split}
\end{align}
Now note that the right-hand side of (\ref{gjfgjkk}) is well-defined and analytic in a punctured neighborhood of $s=1$ (even if $\Re s>1$) for $r>0$. If $\Re s>1$, then the left-hand side (and, therefore, the right-hand side) of (\ref{gjfgjkk}) is continuous for $r\ge 0$. As a corollary, we have
\begin{equation}
\label{kernel delta s expansion}
h(y)G^{(s)}(x,y)=\frac{1}{4\pi \Gamma(s)}\Big[(r^2/4)^{s-1}\Gamma(1-s)-\frac{1}{1-s}\Big]+\mathcal{K}_0(s,x,y)+\mathcal{K}_1(s,x,y),
\end{equation}
where 
$$\mathcal{K}_0(s,x,y)=\frac{\psi_0(x,y)}{4\pi \Gamma(s)}\int_0^{1}{\rm exp}\big(-r^2/4t\big)t^{s-2}dt-\frac{(r^2/4)^{s-1}}{4\pi \Gamma(s)}\int_{0}^{r^2/4}\big(e^{-u}-1\big)u^{-s}du.$$
Here
\begin{itemize}
\item the equality is valid for $r>0$ and any $s$ close to $s=1$;
\item for $\Re s>1$, the left-hand side is continuous at $x=y$;
\item for $x\ne y$, the left-hand side is analytic in $s\in\mathbb{C}$;
\item $\mathcal{K}_1(s,x,y)$ is analytic in $s$ near $s=1$ for any $x,y\in X$ and is continuous in $x,y\in X$;
\item $\mathcal{K}_0(s,x,y)$ is analytic in $s\in\mathbb{C}$ for $x\ne y$ and, due to (\ref{gjfgjkk}), $\mathcal{K}_0(s,x,y)\to 0$ as $r\to 0$ uniformly with respect to $s$ close to $s=1$ (including $s=1$).
\end{itemize}
Let $\zeta^{(r)}(1|\Delta)$ is given by (\ref{reg zeta 1}). In view of (\ref{kernel delta s expansion}) and the identity
$$\lim\limits_{s\to 1}\frac{1-1/\Gamma(s)}{1-s}=\gamma,$$
we have
\begin{equation}
\label{ADM}
\begin{split}
\zeta^{(r)}(1|\Delta)=\lim_{\substack{s\to 1 \\ \Re s>1}}\int_{X}\lim_{x\to y}\Big(h(x)G^{(s)}(y,x)-\frac{1}{4\pi(s-1)}\Big)dS_\rho(y)=\\
=\int_{X}(\mathcal{K}_1(1,y,y)+\gamma/4\pi)dS_\rho(y).
\end{split}
\end{equation}
At the same time, we have
\begin{equation}
\label{Robin}
\begin{split}
m(y)=\lim_{x\to y}\big[h(y)G(x,y)+\frac{1}{2\pi}{\rm log}r\big]=\lim_{x\to y}\Big[\lim_{\substack{s\to 1 \\ \Re s>1}}[h(y)G^{(s)}(x,y)]+\frac{1}{2\pi}{\rm log}r\Big]=\\
=\lim_{x\to y}\Big[\lim_{\substack{s\to 1 \\ \Re s>1}}\Big[\frac{1}{4\pi \Gamma(s)}\Big[(r^2/4)^{s-1}\Gamma(1-s)-\frac{1}{1-s}\Big]\Big]+\frac{1}{2\pi}{\rm log}r\Big]+\mathcal{K}_1(y,y,1)=\\
=\frac{2{\rm log}2-\gamma}{4\pi}+\mathcal{K}_1(y,y,1)
\end{split}
\end{equation}
due to the asymptotics
$$\Gamma(z)-\frac{1}{z}=-\gamma+O(z), \ z\to 0.$$
Comparing (\ref{ADM}) with (\ref{Robin}), one arrives at (\ref{ADM via Robin}).

\section{Evolution of the scalar Robin mass under Ricci flow.} 
\label{Sec Rissi}
\paragraph{Calculation of the scalar Robin mass.} Denote by $m^{(sc)}$ the Robin mass associated with scalar Laplacian $\Delta^{(sc)}=-4\rho^{2}\partial\overline{\partial}$ on $X$. In what follows, we denote by 
$$\langle f\rangle=\frac{1}{{\rm Area}(X;\rho)}\int\limits_{X}f(x)dS_\rho(x)$$
the average value of the function $f$ on $(X,\rho)$. 

Integrating both sides of (\ref{Green Verlinde}) over $X$ and taking into account that the scalar Green function $G^{(sc)}(x,\cdot)$ is $L_2$-orthogonal to constants, we obtain
\begin{align}
\label{average robin sc}
\begin{split}
m^{(sc)}(x)+\langle m^{(sc)}\rangle=-\frac{2}{{\rm Area}(X;\rho)}\int\limits_{X}\Phi(x,y)dS_\rho(y),\\
\langle m^{(sc)}\rangle=-\frac{1}{{\rm Area}(X;\rho)^2}\int\limits_{X}\int\limits_{X}\Phi(x,y)dS_\rho(y)dS_\rho(x),
\end{split}
\end{align}
where $\Phi$ is given by (\ref{Phi func}). Comparing the last two formulas yields
\begin{align}
\label{scalar Robin}
m^{(sc)}(x)=&\frac{1}{{\rm Area}(X;\rho)^2}\int\limits_{X}\int\limits_{X}\Phi(x,y)dS_\rho(y)dS_\rho(x)-\frac{2}{{\rm Area}(X;\rho)}\int\limits_{X}\Phi(x,y)dS_\rho(y).
\end{align}
In addition, from (\ref{Green Verlinde}) and (\ref{Laplace of Phi}) it easily follows that
\begin{equation}
\label{laplacian of robin mass}
\Delta^{(sc)}m^{(sc)}=2\Delta^{(sc)} [G^{(sc)}(x,\cdot)-\Phi(x,\cdot)]=-\frac{2}{{\rm Area}(X,\rho)}+\frac{K}{2\pi}+2\rho^2\overline{\vec{v}}^t(\Im\mathbb{B})^{-1}\vec{v}
\end{equation}
(cf. Proposition 2.3, \cite{Okikiolu} for the case of the Bergman metric).

\paragraph{Evolution of the average Robin mass under Ricci flow: scalar case.} Consider the normalized Ricci flow $t\mapsto \rho_t^{-2}|dz|^2$ of the metrics on $X$,
\begin{equation}
\label{Ricci flow}
\frac{\dot{\rho}_t}{\rho_t}= K_t-\langle K_t\rangle,
\end{equation}
where $K_t=[4\rho^2\partial\overline{\partial}{\rm log}\rho]_t$ is the Gaussian curvature and 
$$\langle K_t\rangle=\frac{1}{A_t}\int\limits_X K_tdS_{\rho}, \qquad A_t={\rm Area}(X;\rho_t).$$ 
It is well known that Ricci flow (\ref{Ricci flow}) preserves the surface area $A_t=A$. In view of the Gauss–Bonnet theorem, we have $A\langle K_t\rangle=2\pi\chi(X)$, where $\chi(X)$ is the Euler characteristic of $X$. As is well known (see \cite{Ham,Chow}), the metric $\rho_t$ converges to the metric of constant curvature $K_\infty=2\pi\chi(X)A^{-1}$ as $t\to+\infty$. 

Denote by $m^{(sc)}_t$ the Robin mass associated with the scalar Laplacian $\Delta^{(sc)}_t=-4\rho_t^{2}\partial\overline{\partial}$ on $X$. Differentiating  both sides of (\ref{average robin sc}) with respect to $t$, we obtain
\begin{equation}
\label{ricci ev 1}
A^2\partial_t\langle m^{(sc)}_t\rangle=\int\limits_{X}\int\limits_{X}\left[2\Phi_t(x,y)\Big[\frac{\dot{\rho}_t(x)}{\rho_t(x)}+\frac{\dot{\rho}_t(y)}{\rho_t(y)}\Big]-\dot{\Phi}_t(x,y)\right]dS_{\rho_t}(y)dS_{\rho_t}(x).
\end{equation}
In view of (\ref{Phi func}) and the fact that the section $F$ (given by (\ref{F section})) in conformally invariant, we have
$$\dot{\Phi}_t(x,y)=\frac{1}{4\pi}\Big[\frac{\dot{\rho}_t(x)}{\rho_t(x)}+\frac{\dot{\rho}_t(y)}{\rho_t(y)}\Big].$$
Then
$$\int\limits_{X}\int\limits_{X}\dot{\Phi}_t(x,y)dS_{\rho_t}(y)dS_{\rho_t}(x)=\frac{1}{2\pi}\int\limits_{X}\int\limits_{X}\frac{\dot{\rho}_t(x)}{\rho_t(x)}dS_{\rho_t}(y)dS_{\rho_t}(x)=-\frac{A_t\dot{A}_t}{4\pi}=0$$
and formulas (\ref{ricci ev 1}), (\ref{Ricci flow}), (\ref{average robin sc}) and the symmetry of $\Phi(x,y)=\Phi(y,x)$ imply
\begin{align*}
\frac{1}{2}\partial_t\langle m^{(sc)}_t\rangle&=2\int\limits_{X}(K_t(x)-K_{\infty})\int\limits_{X}\Phi_t(x,y)\frac{dS_{\rho_t}(y)dS_{\rho_t}(x)}{A^2}=\\
=&\int\limits_{X}\big(K_{\infty}-K_t(x)\big)\big(m^{(sc)}_t(x)+\langle m^{(sc)}_t\rangle\big)\frac{dS_{\rho_t}(x)}{A}=K_{\infty}\langle m^{(sc)}_t\rangle-\int\limits_{X}K_tm^{(sc)}_t\frac{dS_{\rho_t}}{A}.
\end{align*}
Due to (\ref{laplacian of robin mass}), we have
\begin{align}
\label{Ricci evolution of RM}
\Big(\partial_t-2K_{\infty}+\frac{8\pi}{A}\Big)\langle m^{(sc)}_t\rangle=\frac{8\pi}{A}\int\limits_{X}\rho_t^2[\overline{\vec{v}}^t(\Im\mathbb{B})^{-1}\vec{v}]m^{(sc)}_t dS_{\rho_t}-\frac{4\pi}{A}\int\limits_{X}\Delta^{(sc)}m^{(sc)}\cdot m^{(sc)}_t dS_{\rho_t}.
\end{align}
If $X$ is the Riemann sphere $S$ then the first integral in the right-hand side is absent and $K_\infty=4\pi/A$. Then the last formula can be rewritten as
\begin{align*}
\partial_t\langle m^{(sc)}_t\rangle=-\frac{4\pi}{A}\int\limits_{S}\Delta^{(sc)}m^{(sc)}_t\cdot m^{(sc)}_t dS_{\rho_t}.
\end{align*}
Since the scalar Laplacian is non-negative and ${\rm Ker}\,\Delta^{(sc)}=\{{\rm const}\}$, we have 
$$\partial_t\langle m^{(sc)}_t\rangle\le 0,$$ 
where the equality is attained only if $m^{(sc)}_t$ is constant on $S$. Thus, if the area of $S$ is constant, then $\langle m^{(sc)}\rangle$ (as a functional on the space of smooth metrics with given area on $S$) attains its global minimum at the metric of constant curvature. Indeed, let $\Delta_{(sc),0}$ be the laplacian on $S$ corresponding to any metric $\rho^{-2}_0|dz|^2$ of non-constant curvature. Introduce the the family of laplacians $t\mapsto \Delta_{(sc),t}$ ($t\ge 0$) corresponding to Ricci flow (\ref{Ricci flow}). Then the function $t\mapsto\langle m^{(sc),t}\rangle$ decreases. Since the Ricci flow (\ref{Ricci flow}) converges to the metric $\rho^{-2}_\infty|dz|^2$ of constant curvature on $S$, formula (\ref{scalar Robin}) implies $\langle m^{(sc),t}\rangle\to \langle m^{(sc),\infty}\rangle$, where $\Delta^{(sc),\infty}$ is the laplacian corresponding to the metrics $\rho^{-2}_\infty|dz|^2$ of constant curvature. In particular, we obtain $\langle m^{(sc),0}\rangle\ge \langle m^{(sc),\infty}\rangle$. 

Thus, by means of (\ref{ADM via Robin}), we recover the well-known result of Morpurgo (see \cite{Mor}, formula (4)) stating that $\zeta^{(r)}(1|\Delta^{(sc)})$ (as a functional on the space of smooth metrics with given area on $S$) attains minimum at the metric of constant curvature on $S$.  

\subsection*{Statements and Declarations}
\paragraph{Competing Interests.} The authors declare that there are no conflicts of interests and competing interests related to the present work.
\paragraph{Data Availability Statement.} Data sharing not applicable to this article as no datasets were generated or analysed during the current study.
\paragraph{Funding.} The research of the first author was supported by NSERC. The research of the second author was supported by Fonds de recherche du Qu\'ebec.

\paragraph{Acknowledgments.} The authors thank the anonymous referee for several valuable improvements of the text.

\end{document}